\documentclass[letterpaper, 12pt]{amsart}
\usepackage{amsthm}
\usepackage{color,graphicx,fancybox,texdraw}
\usepackage[all]{xy}
\usepackage{multicol}

\theoremstyle{plain}
\newtheorem{theorem}{Theorem}
\newtheorem{lemma}[theorem]{Lemma}
\newtheorem{proposition}[theorem]{Proposition}

\theoremstyle{definition}
\newtheorem{definition}[theorem]{Definition}
\newtheorem{example}[theorem]{Example}
\newtheorem{remark}[theorem]{Remark}

\numberwithin{equation}{section}
\numberwithin{theorem}{section}

\newcommand{\id}[1]{\left\langle#1\right\rangle}

\newcommand{\Z}{\mathbb{Z}}

\newcommand{\F}{\mathbb{F}}

\newcommand{\Fx}[1]{{\F_{#1}[x]\over\id{x^2}}}

\newcommand{\Ra}{U_2(\F_2)}
\newcommand{\Rav}[1]{U_2(\F_{#1})}
\newcommand{\Rb}{\frac{\F_2\id{u,v}}{\id{u^3,v^3,vu,u^2-uv,v^2-uv}}}
\newcommand{\Rc}{\frac{\F_2\id{u,v}}{\id{u^3,v^2,vu,u^2-uv}}}
\newcommand{\Rd}{\frac{\F_4[x;\sigma]}{\id{x^2}}}

\newcommand{\Rev}[1]{\left(\begin{array}{cc}\Fx{#1}&0\\\F_#1&\F_#1\\\end{array}\right)}

\newcommand{\Rfv}[1]{\left(\begin{array}{cc}\Fx{#1}&\F_#1\\0&\F_#1\\\end{array}\right)}

\newcommand{\Rgv}[1]{\left(\begin{array}{cc}\F_#1&\F_{#1}^2\\0&\F_#1\\\end{array}\right)}

\newcommand{\Rhv}[1]{\left\{\left(\begin{array}{cccc}a&0&d&0\\0&a&0&0\\0&0&b&0\\0&c&0&b\\\end{array}\right)|a,b,c,d\in\F_#1\right\}}

\newcommand{\Rj}{\frac{\Z_4\id{u,v}}{\id{u^3,v^3,vu,u^2-uv,v^2-uv,2-uv,2u,2v}}}
\newcommand{\Rk}{\frac{\Z_4\id{u,v}}{\id{u^3,v^2,vu,u^2-uv,2-uv,2u,2v}}}

\newcommand{\Rlv}[1]{\left(\begin{array}{cc}\Z_{#1^2}&0\\\F_#1&\F_#1\\\end{array}\right)}

\newcommand{\Rmv}[1]{\left(\begin{array}{cc}\Z_{#1^2}&\F_#1\\0&\F_#1\\\end{array}\right)}
\newcommand{\Rn}{M_2(\F_2)}
\newcommand{\Rnv}[1]{M_2(\F_{#1})}
\newcommand{\Rmine}{\frac{\F_2\id{u,v}}{\id{u^3,v^3,u^2+v^2+vu,vu^2+uvu+vuv}}}

\newcommand{\diaA}{~\xymatrixrowsep{1pc}\xymatrixcolsep{1pc}\xymatrix{
\textrm{red.}\ar@{=>}[r]&\textrm{symm.}\ar@{=>}[r]&\textrm{rev.}\ar@{=>}[d]\\
\textrm{comm.}\ar@{=>}[ur]\ar@{=>}[r]&\textrm{r.duo}\ar@{=>}[r]&\textrm{semicomm.}\ar@{=>}[r]&\textrm{PS~I}\ar@{=>}[r]&\textrm{2-pmal}\ar@{=>}[r]&\textrm{NI}}}

\newcommand{\diaB}{~\xymatrixrowsep{1pc}\xymatrixcolsep{1pc}\xymatrix{
&&\textrm{reflexive}\ar@{-}[dr]\\
\textrm{red.}\ar@{=>}[r]&\textrm{symm.}\ar@{=>}[r]&\textrm{rev.}\ar@{=>}[d]\ar@{=>}[u]&+\ar@{=>}[l]\\
\textrm{comm.}\ar@{=>}[ur]\ar@{=>}[r]&\textrm{r.duo}\ar@{=>}[r]&\textrm{semicomm.}\ar@{=>}[r]\ar@{-}[ur]&\textrm{abelian}\ar@{=>}[r]&\textrm{PS~I}\ar@{=>}[r]&\textrm{2-pmal}\ar@{<=>}[r]&\textrm{NI}}}

\newcommand{\diaC}{~\xymatrixrowsep{1pc}\xymatrixcolsep{1pc}\xymatrix{
&&&\textrm{reflexive}\ar@{-}[dr]\\
\textrm{reduced}\ar@{=>}[r]&\textrm{commutative}\ar@{=>}[r]\ar@{=>}[dr]&\textrm{symmetric}\ar@{=>}[r]&\textrm{reversible}\ar@{=>}[d]\ar@{=>}[u]&+\ar@{=>}[l]\\
&&\textrm{duo}\ar@{=>}[r]&\textrm{semicommutative}\ar@{=>}[r]\ar@{-}[ur]&\textrm{abelian}\ar@{=>}[r]&\textrm{NI}
}}

\newcommand{\diaD}{~\xymatrixrowsep{1pc}\xymatrixcolsep{1pc}\xymatrix{
&&&&\textrm{refl.}\ar@{-}[d]\\
\textrm{red.}\ar@{=>}[r]&\textrm{comm.}\ar@{=>}[r]\ar@{=>}[dr]&\textrm{symm.}\ar@{=>}[r]&\textrm{rev.}\ar@{=>}[dr]\ar@{=>}[ur]&+\ar@{=>}[l]\\
&&\textrm{duo}\ar@{=>}[rr]&&\textrm{s.comm.}\ar@{=>}[r]\ar@{-}[u]&\textrm{abel.}\ar@{=>}[r]&\textrm{NI}
}}

\newcommand{\dia}{~\xymatrixrowsep{1pc}\xymatrixcolsep{1pc}\xymatrix{
&&\textrm{refl.}\ar@{-}[r]&+\ar@{=>}[dl]\\
\textrm{red.}\ar@{=>}[r]\ar@{=>}@/_/[dr]_{finite}&\textrm{symm.}\ar@{=>}[r]&\textrm{rev.}\ar@{=>}[u]\ar@{=>}[r]&\textrm{semicomm.}\ar@{-}[u]\ar@{=>}[r]&\textrm{PS~I}\ar@{=>}[r]\ar@{<=}@/^2pc/[r]^{local}&\textrm{2-pmal}\ar@{=>}[r]\ar@{<=}@/^2pc/[r]^{artinian}&\textrm{NI}\\
&\textrm{comm.}\ar@{=>}[u]\ar@{=>}[r]&\textrm{r.duo}\ar@{=>}[ur]
}}

\title[{Some Minimal Rings Related to 2-Primal Rings}]{Some Minimal Rings Related to\\ 2-Primal Rings}
\author{Steve Szabo}
\address{Department of Mathematics and Statistics, Eastern Kentucky University, Richmond, KY 40475}
\email{steve.szabo@eku.edu}

\begin{document}

\begin{abstract}
In a paper on the taxonomy of 2-primal rings, examples of various types of rings that are related to commutativity such as reduced, symmetric, duo, reversible and PS~I were given in order to show that the ring class inclusions were strict. Many of the rings given in the examples were infinite. In this paper, where possible, examples of minimal finite rings of the various types are given. Along with the rings in the previous taxonomy, NI, abelian and reflexive rings are also included.\bigskip

\noindent Keywords: finite rings, reversible rings, semicommutative rings, abelian rings, symmetric rings, minimal rings, NI rings, reflexive rings, reduced rings, duo rings, 2-primal rings
\end{abstract}

\maketitle

\section{Introduction and Overview}
As finite structures, in particular finite rings, become more and more prevalent and useful in various disciplines, having minimal examples becomes more important. In coding theory for instance, the most general class of rings that are useful are finite Frobenius rings. This was justified by Wood in \cite{wood_1999}. So, it is important to be aware of small Frobenius rings.  Another reason for minimal examples of various types of rings is simply to have more tangible examples for understanding. In the case of 2-primal rings, it is helpful to know that $\Ra$ is 2-primal but $\Rn$ is not.

In \cite{marks_2002}, the taxonomy of 2-primal rings, Marks states that, in reaction to a question by T. Y. Lam, he provided an example of a finite reversible nonsymmetric ring, namely $\F_2Q_8$, the $\F_2$ group algebra over the quaternion group. This prompted his ``funny little problem'' as he puts it. Is $\F_2Q_8$ a minimal reversible nonsymmetric ring? In \cite{szabo_2017}, the present author shows that $\F_2Q_8$ is indeed such a ring. It is also shown that
\[
\Rmine
\]
is a minimal reversible nonsymmetric ring as well. The two stated examples differ in that fact that $\F_2Q_8$ is right duo but the other ring is not. Hence minimality is independent of being duo. Finding minimal rings with a given property is many times not a trivial matter. Since there is no full characterization of finite rings and as of yet there is no way of listing rings of a particular order, finding rings  with particular properties especial minimal ones is a tedious process.

Minimal rings of various types have been found in recent years: in \cite{martinez_2015} minimal commutative Frobenius nonchain rings, in \cite{xue_1992} minimal noncommutative right duo rings, in \cite{xu_1998} minimal noncommutative semicommutative, in \cite{kim_2011} noncommutative reversible and reflexive rings. Interestingly enough, all such minimal rings are of order 16. Some other well known minimal rings are for instance, $\Fx{2}$ and $\Z_4$, the smallest chain rings that are not fields, $\Rd$, the smallest noncommutative chain ring,  $\Ra$, the smallest noncommutative ring.

A finite ring with identity is a direct sum of rings of prime power order for distinct primes. Representations of finite rings can be found in \cite{wilson_1974}. There have been full classifications of finite rings of order $p^n$ where $p$ is prime and $n\leq 5$: $p^3$ in \cite{gilmer_1973}, commutative $p^4$ in \cite{martinez_2015}, noncommutative $p^4$ in \cite{derr_1994}, nonlocal $p^5$ in \cite{corbas_2000} and local $p^5$ in \cite{corbas_2000_2}. The general classification for $n\geq 6$ is still open although there was some initial work done in \cite{beiranvand_2013} for $n=6$.

In \cite{marks_2003} there is an extensive set of references provided for the various types of rings covered in the taxonomy. For background on NI rings the reader is refered to \cite{hwang_2006,marks_2001} and for background on reflexive rings to \cite{kim_2011,kim_2003}. Although names for most of the ring types mentioned have been settled in the literature, semicommutative rings appear under alternative names. For instance in \cite{marks_2003}, Marks refers to semicommutative rings as rings having the S~I property. They have also been called zero insertive rings.

In Section \ref{sect_prelim}, definitions of the various ring types discussed are given as well as the ring class inclusions. It also covers some basic results on these ring types some known minimal rings needed. Section \ref{sect_main} has the main results on minimal rings of various types. It is split between into two subsections, one on reflexive rings and one on nonreflexive rings.

\section{Preliminaries}
\label{sect_prelim}
In this paper a ring has unity unless otherwise stated. Given a ring $R$, $J(R)$ is the Jacobson radical of $R$, $N(R)$ is the set of nilpotent elements of $R$, $\textrm{Nil}_*(R)$ is the lower nil radical of $R$ (the prime radical of $R$, the intersection of all prime ideals of $R$) and $\textrm{Nil}^*(R)$ is the upper nil radical of $R$ (the sum of all nil ideals of $R$). It is well known that $\textrm{Nil}_*(R)$ is a nil ideal and that $\textrm{Nil}^*(R)$ is the unique largest nil ideal of $R$. So, $\textrm{Nil}_*(R)\subset\textrm{Nil}^*(R)\subset N(R)$.

\begin{definition}
\label{def}
A ring $R$ is ...
\begin{enumerate}
\item{\it reduced} if $N(R)=0$ (equivalently, for all $a\in R$, $a^2=0$ implies $a=0$).
\item{\it symmetric} if for all $a,b,c\in R$, $abc=0$ implies $bac=0$.
\item{\it reversible} if for all $a,b\in R$, $ab=0$ implies $ba=0$.
\item{\it semicommutative} if for all $a,b\in R$, $ab=0$ implies $aRb=0$.
\item{\it reflexive} if for all $a,b\in R$, $aRb=0$ implies $bRa=0$.
\item{\it right (resp. left) duo} if for all $a,b\in R$, $ba\in aR$ ($ba\in Rb$) (equivalently, every right (left) ideal of $R$ is 2-sided). A ring that is both right and left duo is simply a {\it duo} ring.
\item{\it abelian} if each idempotent of $R$ is central.
\item said to satisfy {\it (PS I)} if for every $a\in R$, $R/\textrm{ann}^R_r(aR)$ is 2-primal.
\item{\it 2-primal} if $N(R)=\textrm{Nil}_*(R)$.
\item{\it NI} if $N(R)=\textrm{Nil}^*(R)$ (equivalently, $N(R)\lhd R$).
\end{enumerate}
\end{definition}

Of course a search for minimal rings only needs to consider finite rings. It turns out that in the class of finite rings, there is no distinction between the subclasses of PS~I, 2-primal and NI rings.

\begin{lemma}
\label{lemma_eq}
Let $R$ be a finite ring. Then $\textrm{Nil}_*(R)=\textrm{Nil}^*(R)=J(R)\subset N(R)$. Furthermore, the following are equivalent.
\begin{enumerate}
\item $R$ satisfies PS~I
\item $R$ is 2-primal
\item $R$ is NI
\end{enumerate}
\end{lemma}
\begin{proof}
It is well known that $\textrm{Nil}_*(R)$ is a nil ideal and that $\textrm{Nil}^*(R)$ is the unique largest nil ideal of $R$. So, $\textrm{Nil}_*(R)\subset\textrm{Nil}^*(R)\subset N(R)$. From the definitions it is clear then that PS~I implies 2-primal which implies NI. By Proposition 10.27 in \cite{lam_2001}, $\textrm{Nil}_*(R)=\textrm{Nil}^*(R)=J(R)\subset N(R)$. So, if $R$ is NI then $R$ is 2-primal. Since $R$ is semilocal ($R/J(R)$ is semisimple) and $J(R)$ is nil ($J(R)\subset N(R)$), by Proposition 3.15 of \cite{marks_2003} which says a 2-primal semilocal ring with nil jacobson radical satisfies PS~I, if $R$ is 2-primal it satisfies PS~I.
\end{proof}

From now on, NI rings will be used when any of the 3 types in the previous lemma are needed.

\begin{lemma}[Proposition 3 in \cite{xue_1991}]
\label{lemma_duo}
A finite ring is right duo if and only if it is left duo.
\end{lemma}

\begin{lemma}
A finite reduced ring is commutative.
\end{lemma}
\begin{proof}
From Theorem 12.7 in \cite{lam_2001}, a reduced ring is subdirect product of domains. Since a finite domain is a field, a finite reduced ring is commutative.
\end{proof}

\begin{lemma}
\label{lemma1}
A semicommutative ring is abelian.
\end{lemma}
\begin{proof}
Let $R$ be a semicommutative ring and $e$ an idempotent of $R$. Then $e(1-e)=0$ and $(1-e)e=0$. Since $R$ is semicommutative, $eR(1-e)=0$ and $(1-e)Re=0$. So, for $a\in R$,
\[
ea=ea(1-e+e)=eae=(1-e+e)ae=ae
\] and $e$ is central.
\end{proof}

\begin{lemma}
\label{lemma_j3}
Let $R$ be a local ring. If $R/J(R)$ is a prime field and $J(R)^3=0$ then $R$ is semicommutative.
\end{lemma}
\begin{proof}
Assume $J(R)^3=0$. Let $a,b\in R$ and assume $ab=0$. Then $a,b\in J(R)$. Let $r\in R$. Now, $r=\alpha+r'$ where $\alpha\in R/J(R)$ and $r'\in J(R)$. So,
$arb=a(\alpha+r')b=a\alpha b+ar'b=\alpha ab+ar'b=0$ since $\alpha\in Z(R)$ and $ar'b\in J(R)^3=0$. Hence, $R$ is semicommutative.
\end{proof}

\begin{lemma}
\label{lemma2}
A ring is reversible if and only if  it is semicommutative and reflexive.
\end{lemma}
\begin{proof}
Clearly, a reversible ring is semicommutative and reflexive. Let $R$ be semicommutative and reflexive. Let $a,b\in R$ and assume $ab=0$. By semicommutativity,  $aRb=0$. By reflexivity, $bRa=0$. So, $ba=0$ and $R$ is reversible.
\end{proof}

In \cite{corbas_2000} it was shown that given a finite ring $R$ there exists an orthogonal set of idempotents $\{e_1,\dots,e_m\}\subset R$ such that $e_1+\dots+e_m=1$, $R_i=e_iRe_i$ is a primary ring (a full matrix ring over a local ring) and for $i\neq j$, $M_{ij}=e_iRe_j$ is an $(R_i,R_j)$-bimodule contained in $J(R)$. Furthermore, letting $S=\oplus_iR_i$ and $M=\oplus_{i\neq j}M_{ij}$ then $R=S\oplus M$, $S$ is a subring of $R$ and $M$ is an additive subgroup of $J(R)$ which is an $S$-bimodule. For the purposes of the work herein, this will simply be called the {\it decomposition} of the finite ring $R$. Clearly, $J(R)=J(S)\oplus M$. Notice, if $M^2=0$ then multiplication in $R$ can be modeled as

\begin{equation}
\label{equ_mult}
(s,u)(t,v)=(st,su+vt).
\end{equation}

\begin{lemma}
\label{lemma_atame}
Let $R$ be a finite ring with ring decomposition $\{e_1,\dots,e_m\}$, $R_i=e_iRe_i$, $M_{ij}=e_iRe_j$, $S=\oplus_iR_i$ and $M=\oplus_{i\neq j}M_{ij}$. Then
\begin{enumerate}
\item $R$ is abelian if and only if $M=0$ and $R_i$ is local for $i\in\{1,\dots,m\}$.\label{lemma_atame_1}
\item $R$ is NI if and only if $R_i$ is local for $i\in\{1,\dots,m\}$.\label{lemma_atame_2}
\item If $R$ is abelian then $R$ is NI.\label{lemma_atame_4}
\item If $M^2=0$ and $M\neq 0$ then $R$ is nonreflexive.\label{lemma_atame_5}
\end{enumerate}
\end{lemma}
\begin{proof}~

\noindent (\ref{lemma_atame_1}) This follows directly from the definition of an abelian ring.

\noindent (\ref{lemma_atame_2}) Assume $R_1$ is not local. Since $R_1$ is a primary ring, $R_1=M_n(T)$ for $n>1$ and some local ring $T$. Let $e_{ij}$ be the $(i,j)$ matrix unit of $R_1$. Then $e_{12}$ and $e_{21}$ are nilpotent but $e_{12}+e_{21}$ is not. So, $R$ in this case is not NI.

Now assume $R_i$ is local for $i\in\{1,\dots,m\}$. Then $N(S)=J(S)$. By Lemma \ref{lemma_eq}, $J(R)\subset N(R)$. Let $r\in N(R)$. Then $r=s+u$ for some $s\in S$ and $u\in M\subset J(R)$ and for some $n$,
$0=(s+u)^n=s^n+t$ for some $t\in J(R)$ since $M\subset J(R)$. So, $s^n=-t$ and then $s^n\in J(R)\cap S=N(S)$. Hence, $s\in N(S)=J(S)$ and $r\in J(R)$ showing $N(R)=J(R)$ and $R$ is NI.

\noindent (\ref{lemma_atame_4}) Follows from (\ref{lemma_atame_1}) and (\ref{lemma_atame_2}).

\noindent (\ref{lemma_atame_5}) Assume $M^2=0$ and $M\neq 0$. Renumbering if necessary, it may be assumed that $M_{12}\neq 0$. Since $M^2=0$, multiplication in $R$ can be modeled as in \ref{equ_mult}. Let $d\in M_{12}\setminus 0$ and $(s,u)\in R$. Then
\[
(0,d)(s,u)(e_1,0)=(0,d)(se_1,ue_1)=(0,0).
\]
showing $(0,d)R(e_1,0)=0$. But,
\[
(e_1,0)(e_1,0)(0,d)=(0,d)
\]
showing $(e_1,0)R(0,d)\neq 0$. Hence, $R$ is nonreflexive.
\end{proof}

Figure \ref{fig_1} has the ring class inclusions of the types of  finite rings being considered. The implications not covered in the lemmas above follow directly from the definitions.

\begin{figure}[!htp]
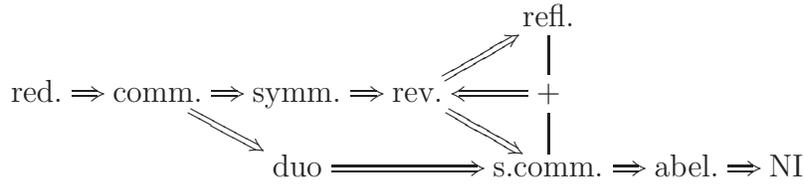

\label{fig_1}
\diaD
\caption{Finite Rings}
\end{figure}

It is easily seen that if * is one of the ring types being considered, a finite ring is of type * if and only if every ring direct summand is of type *. This is why when searching for minimal rings of these types, only indecomposable rings need to be checked. The only exception is a non-NI nonreflexive ring (see Theorem \ref{theo_nninref}). Since a finite ring with identity is a direct sum of rings of prime power order for distinct primes, the search is further limited to rings of prime power order.

\begin{proposition}[\cite{eldridge_1968}]
\label{prop_eld}
For prime $p$, there are no noncommutative rings of order $p$ or $p^2$ and the only noncommutative ring of order $p^3$ is $\Rav{p}$.
\end{proposition}

\begin{proposition}[\cite{derr_1994} or Corollary 1.10 in \cite{corbas_2000}]
\label{prop_nc16}
Let $p$ be prime. The noncommutative nonlocal indecomposable rings of order $p^4$ are
\begin{itemize}
\begin{multicols}{2}
\item $\Rfv{p}$
\item $\Rmv{p}$
\item $\Rgv{p}$
\item $\Rev{p}$
\item $\Rlv{p}$
\item $\Rnv{p}$
\end{multicols}
\item $\Rhv{p}$
\end{itemize}
\end{proposition}

\begin{proposition}
\label{prop_2prim}
Let $p$ be prime. There are no non-NI rings of order $p$, $p^2$ or $p^3$. The only non-NI ring(s) of order $p^4$ is $\Rnv{p}$, of order $p^5$ is $\Rnv{p}\oplus\F_p$ and of order $p^6$ are $\Rnv{p}\oplus\F_p\oplus\F_p$, $\Rnv{p}\oplus\Fx{p}$, $\Rnv{p}\oplus\Z_{p^2}$ and $\Rnv{p}\oplus\F_{p^2}$. The minimal non-NI ring is $\Rn$.
\end{proposition}
\begin{proof}
Since commutative rings are NI, by Lemma \ref{lemma_atame}(\ref{lemma_atame_2}) and Proposition \ref{prop_eld}, all rings of order $p$, $p^2$ or $p^3$ are NI. By Lemma \ref{lemma_atame}(\ref{lemma_atame_2}) and Proposition \ref{prop_nc16}, the only ring of order $p^4$ that is not NI is $\Rnv{p}$ and of order $p^5$ is $\Rnv{p}\oplus\F_2$. So, the minimal non-NI ring is $\Rn$. Knowing that $\F_2\oplus\F_2$, $\Fx{2}$, $\Z_4$ and $\F_{p^2}$ are the rings of order 4, by Lemma \ref{lemma_atame}(\ref{lemma_atame_2}), the rings $\Rnv{p}\oplus\F_2\oplus\F_2$, $\Rnv{p}\oplus\Fx{2}$, $\Rnv{p}\oplus\Z_4$ and $\Rnv{p}\oplus\F_{p^2}$ are the only decomposable non-NI rings of order $p^6$. The only possible indecomposable non-NI ring of order $p^6$ would be a ring with components $R_1=\Rnv{p}$ and $R_2=\F_p$ and $|M|=p$ in its ring decomposition. But, the minimal submodules of $\Rnv{p}$ are isomorphic to $\F_p^2$ so no such ring exists.
\end{proof}

\begin{remark}
\label{rem_l16}
The set of noncommutative local rings of order 16 can be deduced from work in \cite{corbas_2000} or \cite{derr_1994}. It turns out that they are precisely the minimal noncommutative semicommutative rings which is the subject of Proposition \ref{prop_semic}. Note $\sigma$ denotes the Frobenius automorphism on $\F_4$.
\end{remark}

\begin{proposition}[\cite{xu_1998} Theorem 8]
\label{prop_semic}
A minimal noncommutative semicommutative ring has order 16. The complete list of such rings is
\begin{itemize}
\item $\Rd$
\item $\Rb$
\item $\Rj$
\item $\Rc$
\item $\Rk$.
\end{itemize}
\end{proposition}

Two other results on rings of order 16 which are necessary for our classification are provided here.

\begin{proposition}[\cite{xue_1992} Theorem 3]
\label{prop_duo}
A minimal noncommutative duo ring has order 16. The complete list of such rings is
\begin{itemize}
\item $\Rd$
\item $\Rb$
\item $\Rj$.
\end{itemize}
\end{proposition}

\begin{proposition}[\cite{kim_2011} Theorem 5]
\label{prop_refl}
A minimal noncommutative reflexive ring has order 16. The complete list of such rings is
\begin{itemize}
\item $\Rd$
\item $\Rn$.
\end{itemize}
\end{proposition}

\section{Minimal Rings}
\label{sect_main}
In this section the minimal rings of the various types according to the diagram in Figure \ref{fig_1} are identified. This will be done first for nonreflexive rings and then for reflexive rings.

\subsection{Minimal Nonreflexive Rings}

In the context of finite nonreflexive rings, the ring class inclusions are
\begin{center}
~\xymatrixrowsep{1pc}\xymatrixcolsep{1pc}\xymatrix{
\textrm{duo}\ar@{=>}[r]&\textrm{semicommutative}\ar@{=>}[r]&\textrm{abelian}\ar@{=>}[r]&\textrm{NI}
}
\end{center}

\begin{theorem}[Minimal Nonreflexive Duo]
A minimal nonreflexive duo ring has order 16. The complete list of such rings is
\begin{itemize}
\item $\Rb$
\item $\Rj$.
\end{itemize}
\end{theorem}
\begin{proof}
Follows from Propositions \ref{prop_duo} and \ref{prop_refl}.
\end{proof}

\begin{theorem}[Minimal Nonreflexive Semicommutative Nonduo]
A minimal nonreflexive semicommutative nonduo ring has order 16. The complete list of such rings is
\begin{itemize}
\item $\Rc$
\item $\Rk$.
\end{itemize}
\end{theorem}
\begin{proof}
Follows from Propositions \ref{prop_semic}, \ref{prop_duo} and \ref{prop_refl}.
\end{proof}

From Remark \ref{rem_l16} and Proposition \ref{prop_eld}, any noncommutative abelian ring of order less than 32 is local semicommutative. Therefore, a nonsemicommutative abelian ring is of order at least 32. The next proposition shows that such a ring must actually be at least of order 64.

\begin{proposition}
\label{prop_semi}
A local ring of order less than 64 is semicommutative.
\end{proposition}
\begin{proof}
Commutative rings are semicommutative. Let $p$ be a prime. By Proposition \ref{prop_eld}, there are no noncommutative local rings of order $p$, $p^2$ or $p^3$. By Remark \ref{rem_l16}, there are no nonsemicommutative local rings of order $p^4$. So, a local ring of order less than 32 is semicommutative. For rings of order 32, by Lemma \ref{lemma_j3}, only the local noncommutative rings where the cube of its jacobson radical is nonzero need to be checked for semicommutativity. Such rings can be found in the classification of rings of order 32 in \cite{corbas_2000_2} which are given in cases 1.2.1, 2.2.b and 2.2.c. There are 7 such rings which are listed here.
\begin{enumerate}
\item $\frac{\F_2\id{u,v}}{\id{u^4,uv,vu-u^3,v^2}}$
\item $\frac{\F_2\id{u,v}}{\id{u^4,uv,vu-u^3,v^2-u^3}}$
\item $\frac{\Z_4\id{u,v}}{\id{u^4,uv,vu-u^3,v^2,u^3-2}}$
\item $\frac{\Z_4\id{u,v}}{\id{u^4,uv,vu-u^3,v^2-u^3,u^3-2}}$
\item $\frac{\Z_4\id{u,v}}{\id{u^4,uv,vu-u^3,v^2,u^2-2}}$
\item $\frac{\Z_4\id{u,v}}{\id{u^4,uv,vu-u^3,v^2-u^3,u^2-2}}$
\item $\frac{\Z_4\id{u,v}}{\id{u^4,uv,vu-u^3,v^2,u^2-2-2u}}$.
\end{enumerate}
Let $R$ be any of the rings above. Then $R$ is local, $J(R)=\F_2u+\F_2v+J^2$, $J^2=\F_2u^2+J^3$, $J^3=Fu^3$ and $J^4=0$. In a finite local ring, the jacobson radical is the set of zero divisors. Let $a,b\in J(R)$ and $c\in R$. Then for some $\alpha_i,\beta_i,\gamma_i,\delta_i,r\in\F_2$, $a=\alpha_1u+\beta_1v+\gamma_1u^2+\delta_1u^3$, $b=\alpha_2u+\beta_2v+\gamma_2u^2+\delta_2u^3$ and $c=r+\alpha_3u+\beta_3v+\gamma_3u^2+\delta_3u^3$. Assume $ab=0$. Then
\begin{eqnarray*}
0&=&ab\\
&=&(\alpha_1u+\beta_1v+\gamma_1u^2+\delta_1u^3)(\alpha_2u+\beta_2v+\gamma_2u^2+\delta_2u^3)\\
&=&\alpha_1\alpha_2u^2+x
\end{eqnarray*}
for some $x\in J(R)^3$ showing $\alpha_1\alpha_2=0$. Now, $acb=arb+\alpha_1\alpha_3\alpha_2u^3=0$. So, if $ab=0$ then $aRb=0$ and $R$ is semicommutative.
\end{proof}

\begin{theorem}[Minimal Nonreflexive Abelian Nonsemicommutative]
A minimal nonreflexive abelian nonsemicommutative ring is of order 64 an example of which is
\[
\frac{\F_2\id{u,v}}{\id{u^2,v^2,uvu-vuv}}
\]
\end{theorem}
\begin{proof}
By Proposition \ref{prop_semi}, an abelian ring of order less than 64 is semicommutative. Let
\[
R=\frac{\F_2\id{u,v}}{\id{u^2,v^2,uvu-vuv}}.
\]
$R$ is local, $J(R)=\F_2u+\F_2v+J^2$, $J^2=\F_2uv+\F_2vu+J^3$, $J^3=\F_2uvu$ and $J^4=0$. So, $|R|=64$. Notice $u^2=0$ but $uvu\neq 0$ showing $R$ is nonsemicommutative.

To see that $R$ is nonreflexive, notice, for $\alpha,\beta,\gamma,\delta,\epsilon,\zeta\in F$,
\[
u(\alpha+\beta u+\gamma v+\delta uv+\epsilon vu+\zeta uvu)uv=\alpha u^2v=0
\]
but $0\neq uvu\in uvRu$. So, $uRuv=0$ but $uvRu\neq 0$ showing $R$ is nonreflexive.
\end{proof}

\begin{remark}
\label{remark_1}
The ring
\[
\frac{\F_2\id{u,v}}{\id{u^2,v^2,uvu-vuv}}
\]
was shown to be a minimal nonreflexive abelian nonsemicommutative ring. From the proof it can be seen that it is actually minimal among all nonsemicommutative rings, reflexive or not.
\end{remark}

\begin{theorem}[Minimal Nonreflexive NI Nonabelian]
\label{theo_minref}
The minimal NI nonabelian nonreflexive ring is $\Ra$ and is of order 8.
\end{theorem}
\begin{proof}
By Proposition \ref{prop_refl}, $\Ra$ is nonreflexive.
Clearly $\Ra$ is the minimal nonabelian ring since it is not abelian and is the unique minimal noncommutative ring. By Lemma \ref{lemma_atame}(\ref{lemma_atame_2}), $\Ra$ is NI.
\end{proof}

\begin{theorem}[Minimal Non-NI Nonreflexive]
\label{theo_nninref}
A minimal non-NI nonreflexive ring has order 128. The complete list of such rings is
\[
R=\left(
  \begin{array}{cc}
    \Rn & \left(\begin{array}{cc}\F_2&0\\\F_2&0 \\\end{array}\right)\\
    0 & \F_2 \\
  \end{array}
\right), ~R^{op}
\textrm{ and }
\Rn\oplus\Ra.
\]
\end{theorem}
\begin{proof}
By Proposition \ref{prop_2prim}, the only non-NI indecomposable rings of order less than 128 are $\Rn$ and $\Rnv{3}$. But, these are both reflexive. By Theorem \ref{theo_minref}, the smallest nonreflexive ring is $\Ra$ which is unique. So, $\Rn\oplus\Ra$ is the unique decomposable minimal non-NI nonreflexive of order 128.

An indecomposable example would have to have at least one of the component rings in its ring decomposition to be nonlocal by Lemma \ref{lemma_atame}. Since the minimal right or left modules of $\Rn$ are isomorphic to $\F_2^2$, the only indecomposable non-NI non-reflexive rings of order 128 are
\[
R=\left(
  \begin{array}{cc}
    \Rn & \left(\begin{array}{cc}\F_2&0 \\\F_2&0 \\\end{array}\right) \\
    0 & \F_2 \\
  \end{array}
\right)
\]
and
\[
R^{op}=\left(
  \begin{array}{cc}
    \Rn & 0\\
    \left(\begin{array}{cc}\F_2&\F_2 \\0&0 \\\end{array}\right) & \F_2 \\
  \end{array}
\right),
\]
which is clearly not isomorphic to $R$. By the uniqueness of the components, it is clear that there are no other such rings.
\end{proof}

\subsection{Minimal Reflexive Rings}
In this section minimal reflexive rings are identified. The following diagram shows the ring class inclusions for finite reflexive rings.

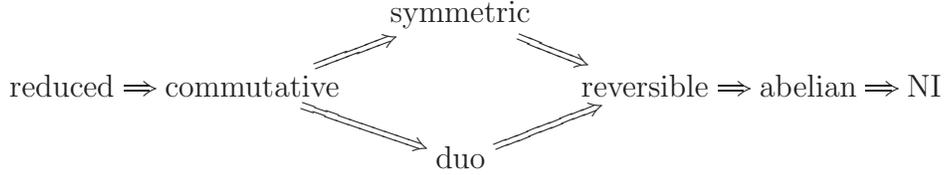
\begin{center}
\xymatrixrowsep{1pc}\xymatrixcolsep{1pc}\xymatrix{
&&\textrm{symmetric}\ar@{=>}[dr]\\
\textrm{reduced}\ar@{=>}[r]&\textrm{commutative}\ar@{=>}[dr]\ar@{=>}[ur]&&\textrm{reversible}\ar@{=>}[r]&\textrm{abelian}\ar@{=>}[r]&\textrm{NI}\\
&&\textrm{duo}\ar@{=>}[ur]}
\end{center}

Of course $\F_2$ is the minimal reduced ring and the minimal nonreduced commutative rings are the indecomposable rings of order 4 that are not fields namely $\Fx{2}$ and $\Z_4$.

\begin{theorem}[Minimal Symmetric Duo Noncommutative]
The minimal symmetric duo noncommutative ring is $\Rd$ and is of order 16.
\end{theorem}
\begin{proof}
Let $R=\Rd$. By Propositions \ref{prop_semic} and \ref{prop_refl}, $R$ is the minimal semicommutative reflexive duo ring. By Lemma \ref{lemma2}, $R$ is reversible. In \cite{szabo_2017} it was shown that a minimal reversible nonsymmetric ring is of order 256.  Since $|R|=16$, $R$ is symmetric. The symetricity of $R$ can easily be shown directly.
\end{proof}

\begin{theorem}[Minimal Symmetric Nonduo]
A minimal symmetric nonduo ring is of order 32 an example of which is
\[
\F_2\id{u,v}\over\id{u^3,v^2,u^2+uv+vu,uvu}.
\]
\end{theorem}
\begin{proof}
By Remark \ref{rem_l16}, the rings listed in Proposition \ref{prop_semic} are the noncommutative local rings of order 16. By inspection, it is clear that none of these rings are reversible hence they are nonsymmetric. For a prime $p$, by Proposition \ref{prop_eld}, any local ring of order $p$, $p^2$ or $p^3$ is commutative. So, any symmetric ring of order less than 32 is commutative and therefore duo.

Let
\[
R={\F_2\id{u,v}\over\id{u^3,v^2,u^2+uv+vu,uvu}}.
\]
First note that any third degree monomial in $u$ and $v$ is 0. It is easy to see then that $R$ is an $\F_2$-algebra with basis $\{1,u,v,uv,vu\}$ and $|R|=32$. Next, $R$ is not left duo since $v(u+v)=vu\notin\{av+buv|a,b\in\F_2\}=Rv$. By Lemma \ref{lemma_duo}, $R$ is not duo. To show $R$ is reversible, assume
\[
(a_1+b_1u+c_1v+d_1uv+e_1vu)(a_2+b_2u+c_2v+d_2uv+e_2vu)=0
\]
for $a_i,b_i,c_i,d_i,e_i\in\F_2$. Since $R$ is local, $a_1=a_2=0$. Then
\begin{eqnarray*}
0&=&(b_1u+c_1v+d_1uv+e_1vu)(b_2u+c_2v+d_2uv+e_2vu)\\
&=&(b_1c_2+b_1b_2)uv+(c_1b_2+b_1b_2)vu
\end{eqnarray*}
showing $(b_1c_2+b_1b_2)=0$ and $(c_1b_2+b_1b_2)=0$. So,
\begin{eqnarray*}
0&=&(b_1c_2+b_1b_2)vu+(c_1b_2+b_1b_2)uv\\
&=&(a_2+b_2u+c_2v+d_2uv+e_2vu)(a_1+b_1u+c_1v+d_1uv+e_1vu)
\end{eqnarray*}
and $R$ is reversible. From \cite{szabo_2017}, a reversible ring of order less than 256 is symmetric. Since $|R|=32$, $R$ is symmetric. Hence, $R$ is a minimal symmetric nonduo ring.
\end{proof}

\begin{theorem}[Minimal Reversible Nonsymmetric Right Duo]
A minimal reversible nonsymmetric right duo ring is of order 256 an example of which is $\F_2Q_8$.
\end{theorem}
\begin{proof}
One of the main results of \cite{szabo_2017}.
\end{proof}

\begin{theorem}[Minimal Reversible Nonsymmetric Nonduo]
A minimal reversible nonsymmetric nonduo ring is of order 256 an example of which is
\[
\frac{\F_2\id{u,v}}{\id{u^3,v^3,u^2+v^2+vu,vu^2+uvu+vuv}}.
\]
\end{theorem}
\begin{proof}
One of the main results of \cite{szabo_2017}.
\end{proof}

\begin{remark}[Reflexive Abelian Nonsemicommutative]
\label{remark_ran}
A minimal reflexive abelian nonsemicommutative ring has turned out to be more elusive than the rest. By Proposition \ref{prop_semi}, a minimal reflexive abelian nonsemicommutative ring is of order at least 64. At the moment, there is no known finite local reflexive nonsemicommutative ring. In \cite{marks_2003}, Lemma 5.3 states that there exists an infinite local ring $R$ that contains a prime ideal $P$ that is not completely prime. So, $R/P$ is a local ring that is prime but not completely prime i.e. a domain, meaning $R/P$ is a prime ring with zero divisors. It is easy to show that a prime ring is reflexive and that prime ring with zero divisors is not semicommutative. Hence, $R/P$ is a local reflexive nonsemicommutative ring. So, such rings do exist. It is important to note that according to \cite{lam_2001} Exercise 13.3, a finite prime ring is a matrix ring over a finite field. Hence, a local finite prime ring is a finite field. This shows there does not exist a finite local prime ring with zero divisors. So, if there does exist a finite reflexive abelian nonsemicommutative ring, it is not a prime ring with zero divisors.
\end{remark}

\begin{example}
\label{ex_nanir}
The following ring is a nonabelian NI reflexive ring of order 64. On the additive group
\[
R=\Fx{2}\oplus\Fx{2}\oplus\F_2\oplus\F_2
\]
define multiplication as
\begin{eqnarray*}
&&(\alpha_1+\beta_1x,\gamma_1+\delta_1x,\epsilon_1,\zeta_1)(\alpha_2+\beta_2t,\gamma_2+\delta_2x,\epsilon_2,\zeta_2)\\
&=&\big(\epsilon_1\zeta_2x+(\alpha_1+\beta_1x)(\alpha_2+\beta_2x),\\
&&\epsilon_2\zeta_1x+(\gamma_1+\delta_1x)(\gamma_2+\delta_2x),\\
&&\alpha_1\epsilon_2+\gamma_2\epsilon_1,\\
&&\gamma_1\zeta_2+\alpha_2\zeta_1\big)
\end{eqnarray*}
for $\alpha_j,\beta_j,\gamma_j,\delta_j,\epsilon_j,\zeta_j\in\F_2$. It is straight forward to check that $R$ is a ring with $1=(1,1,0,0)$. Furthermore, $e_1=(1,0,0,0)$ and $e_2=(0,1,0,0)$ are idempotents which give the ring decomposition of $R$ where the direct summands are $R_1$, $R_2$, $M_{12}$ and $M_{21}$ respectively. First, $R$ is nonabelian since $M\neq 0$. Secondly, since $R_1$ and $R_2$ are local, by Lemma \ref{lemma_atame}(\ref{lemma_atame_2}), $R$ is NI. Finally, it will be shown that $R$ is also reflexive.

Let $a,b\in R$ and assume $aRb=0$. So $a=(a_1,a_2,a_{12},a_{21})$ and $b=(b_1,b_2,b_{12},b_{21})$ for $a_1,a_2,b_1,b_2\in\Fx{2}$ and $a_{12},a_{21},b_{12},b_{21}\in\F_2$. Then
\[
0=a(1,0,0,0)b=(a_1b_1,a_{21}b_{12}x,a_1b_{12},a_{21}b_1),
\]
\[
0=a(0,1,0,0)b=(a_{12}b_{21}x,a_2b_2,a_{12}b_2,a_2b_{21}),
\]
\[
0=a(0,0,1,0)b=(a_{1}b_{21}x,a_{21}b_2x,a_{1}b_2,0)
\]
and
\[
0=a(0,0,0,1)b=(a_{12}b_{1}x,a_{2}b_{12}x,0,a_2b_1).
\]
This shows that $a_1b_1=a_2b_2=0$ and
\[
a_{21}b_{12},a_{12}b_{21},a_{1}b_{21},a_{21}b_2,a_{12}b_{1},a_{2}b_{12},a_1b_{12},a_{21}b_1,a_{12}b_2,a_2b_{21},a_{1}b_2,a_2b_1\in\id{x}.
\]
So, for $r=(r_1,r_2,r_{12},r_{21})\in R$,
\begin{eqnarray*}
bra&=&(b_1,b_2,b_{12},b_{21})(r_1,r_2,r_{12},r_{21})(a_1,a_2,a_{12},a_{21})\\
&=&(b_{1}r_{1}a_{1}+b_{1}r_{12}a_{21}x+b_{12}r_{2}a_{21}x+b_{12}r_{21}a_{1}x,\\
&&b_{2}r_{2}a_{2}+b_{2}r_{21}a_{12}x+b_{21}r_{1}a_{12}x+b_{21}r_{12}a_{2}x,\\
&&b_{1}r_{1}a_{12}+b_{1}r_{12}a_{2}+b_{12}r_{2}a_{2},\\
&&b_{2}r_{2}a_{21}+b_{2}r_{21}a_{1}+b_{21}r_{1}a_{1})\\
&=&0.
\end{eqnarray*}
Hence, $R$ is reflexive.
\end{example}

\begin{theorem}[Minimal NI Nonabelian Reflexive]
A minimal NI nonabelian reflexive ring is of order 64 an example of which is the additive group
\[
R=\Fx{2}\oplus\Fx{2}\oplus\F_2\oplus\F_2
\]
with multiplication as in Example \ref{ex_nanir}.
\end{theorem}
\begin{proof}
For a prime $p$, Proposition \ref{prop_eld} shows that the only indecomposable nonabelian ring of order $p$, $p^2$ or $p^3$ is $\Rav{p}$. But this ring is nonreflexive (see Theorem \ref{theo_minref}). From Lemma \ref{lemma_atame}(\ref{lemma_atame_5}) and Proposition \ref{prop_nc16}, the only indecomposable nonabelian reflexive ring of order $p^4$ is $\Rnv{p}$ since all the others have $M\neq 0$ but $M^2=0$ in their ring decomposition. But $\Rnv{p}$ is non-NI by Proposition \ref{prop_2prim}. In \cite{corbas_2000}, it was shown that there are only 2 rings of order 32 which have a ring decomposition with $M^2\neq 0$ (see Lemma 1.5 in \cite{corbas_2000}). In either one of those two rings, $(0,0,1,0)R(0,0,0,1)\neq 0$ but $(0,0,0,1)R(0,0,1,0)=0$ so they are nonreflexive. By Lemma \ref{lemma_atame}, a nonabelian NI ring with $M^2=0$ (nonabelian NI guarantees $M\neq0$) is nonreflexive. So, all indecomposable nonabelian rings of order 32 are nonreflexive. Hence, there are no NI nonabelian reflexive rings of order less than 64. $R$ was shown to be an NI nonabelian reflexive ring of order 64 in Example \ref{ex_nanir}.
\end{proof}

\begin{theorem}[Minimal Non-NI Reflexive]
The minimal non-NI reflexive rings is $\Rn$.
\end{theorem}
\begin{proof}
By Proposition Proposition \ref{prop_refl}, there are two minimal noncommutative reflexive rings, $\Rd$ and $\Rn$. By Proposition \ref{prop_semic}, $\Rd$ is semicommutative and therefore NI. By Lemma \ref{lemma_atame}(\ref{lemma_atame_2}), $\Rn$ is not NI.
\end{proof}

\section{Conclusion}
Minimal rings of all but one of the types of rings under consideration were found. Reflexive abelian nonsemicommutative, proved more elusive and is left for future consideration. From Proposition \ref{prop_semi}, such a ring is at least of order 64. There is no known example of a finite local reflexive nonsemicommutative ring. It was however shown  in Remark \ref{remark_ran} that an infinite ring of this type exists.

\appendix
\bibliographystyle{amsplain}
\bibliography{../SteveSzaborefs}

\providecommand{\bysame}{\leavevmode\hbox to3em{\hrulefill}\thinspace}
\providecommand{\MR}{\relax\ifhmode\unskip\space\fi MR }
\providecommand{\MRhref}[2]{%
  \href{http://www.ams.org/mathscinet-getitem?mr=#1}{#2}
}
\providecommand{\href}[2]{#2}
\begin{thebibliography}{10}

\bibitem{corbas_2000}
B.~Corbas and G.~D. Williams, \emph{Rings of order {$p^5$}. {I}. {N}onlocal
  rings}, J. Algebra \textbf{231} (2000), no.~2, 677--690. \MR{1778165}

\bibitem{corbas_2000_2}
\bysame, \emph{Rings of order {$p^5$}. {II}. {L}ocal rings}, J. Algebra
  \textbf{231} (2000), no.~2, 691--704. \MR{1778166}

\bibitem{derr_1994}
J.~B. Derr, G.~F. Orr, and Paul~S. Peck, \emph{Noncommutative rings of order
  {$p^4$}}, J. Pure Appl. Algebra \textbf{97} (1994), no.~2, 109--116.
  \MR{1312757 (95m:16012)}

\bibitem{eldridge_1968}
K.~E. Eldridge, \emph{Orders for finite noncommutative rings with unity}, Amer.
  Math. Monthly \textbf{75} (1968), 512--514. \MR{0230772 (37 \#6332)}

\bibitem{gilmer_1973}
Robert Gilmer and Joe Mott, \emph{Associative rings of order {$p^{3}$}}, Proc.
  Japan Acad. \textbf{49} (1973), 795--799. \MR{0369422}

\bibitem{hwang_2006}
Seo~Un Hwang, Young~Cheol Jeon, and Yang Lee, \emph{Structure and topological
  conditions of {NI} rings}, J. Algebra \textbf{302} (2006), no.~1, 186--199.
  \MR{2236599}

\bibitem{beiranvand_2013}
P.~Karimi~Beiranvand, R.~Beyranvand, and M.~Gholami, \emph{Classification of
  finite rings of order {$p^6$} by generators and relations}, J. Math. (2013),
  Art. ID 467905, 8. \MR{3096801}

\bibitem{kim_2011}
Byung-Ok Kim and Yang Lee, \emph{Minimal noncommutative reversible and
  reflexive rings}, Bull. Korean Math. Soc. \textbf{48} (2011), no.~3,
  611--616. \MR{2827769}

\bibitem{kim_2003}
Nam~Kyun Kim and Yang Lee, \emph{Extensions of reversible rings}, J. Pure Appl.
  Algebra \textbf{185} (2003), no.~1-3, 207--223. \MR{2006427}

\bibitem{lam_2001}
T.~Y. Lam, \emph{A first course in noncommutative rings}, second ed., Graduate
  Texts in Mathematics, vol. 131, Springer-Verlag, New York, 2001.
  \MR{MR1838439 (2002c:16001)}

\bibitem{marks_2001}
Greg Marks, \emph{On 2-primal {O}re extensions}, Comm. Algebra \textbf{29}
  (2001), no.~5, 2113--2123. \MR{1837966}

\bibitem{marks_2002}
\bysame, \emph{Reversible and symmetric rings}, J. Pure Appl. Algebra
  \textbf{174} (2002), no.~3, 311--318. \MR{1929410}

\bibitem{marks_2003}
\bysame, \emph{A taxonomy of 2-primal rings}, J. Algebra \textbf{266} (2003),
  no.~2, 494--520. \MR{1995125}

\bibitem{martinez_2015}
Edgar Mart{\'{\i}}nez-Moro and Steve Szabo, \emph{On codes over local
  {F}robenius non-chain rings of order 16}, Noncommutative Rings and Their
  Applications (Steven Dougherty, Alberto Facchni, Andr\'e Leroy, Edmund
  Puczylowski, and Patrick Sol\'e, eds.), Contemp. Math., vol. 634, Amer. Math.
  Soc., Providence, RI, 2015, pp.~227--243.

\bibitem{szabo_2017}
Steve Szabo, \emph{Minimal reversible nonsymmetric rings}, submitted (2017).

\bibitem{wilson_1974}
Robert~S. Wilson, \emph{Representations of finite rings}, Pacific J. Math.
  \textbf{53} (1974), 647--649. \MR{0369423}

\bibitem{wood_1999}
Jay~A. Wood, \emph{Duality for modules over finite rings and applications to
  coding theory}, Amer. J. Math. \textbf{121} (1999), no.~3, 555--575.
  \MR{MR1738408 (2001d:94033)}

\bibitem{xu_1998}
Liqiong Xu and Weimin Xue, \emph{Structure of minimal non-commutative
  zero-insertive rings}, Math. J. Okayama Univ. \textbf{40} (1998), 69--76
  (2000). \MR{1755920 (2001b:16036)}

\bibitem{xue_1991}
Weimin Xue, \emph{On strongly right bounded finite rings}, Bull. Austral. Math.
  Soc. \textbf{44} (1991), no.~3, 353--355. \MR{1138010}

\bibitem{xue_1992}
\bysame, \emph{Structure of minimal noncommutative duo rings and minimal
  strongly bounded nonduo rings}, Comm. Algebra \textbf{20} (1992), no.~9,
  2777--2788. \MR{1176837 (93i:16003)}

\end{thebibliography}

\end{document}